\documentclass[a4paper, twoside, 11pt]{article}

\usepackage{geometry}

\usepackage[english]{babel}
\usepackage[utf8]{inputenc}

\usepackage{mathpazo}
\usepackage{authblk}
\usepackage{mathrsfs}
\usepackage{amsfonts, amssymb, amsmath, amsthm}
\usepackage{amsbsy}
\usepackage{enumitem}
\usepackage{stmaryrd}
\usepackage[backref=page]{hyperref}

\usepackage{graphicx}
\usepackage{tikz-cd}
\usepackage{amscd}
\usepackage{tabularx}

\usepackage{fancyhdr}
\usepackage{ifthen}
\newtheoremstyle{mystyle}{}{}{\rmfamily}%
{}{\normalfont\bfseries}{ }{ }{}

\hypersetup{colorlinks=true, linkcolor=blue, citecolor=blue, urlcolor=blue}

\newcommand{\R}{\mathbb{R}}
\newcommand{\N}{\mathbb{N}}
\newcommand{\Z}{\mathbb{Z}}
\newcommand{\Pa}{\mathcal{P}}

\newtheorem{theorem}{Theorem}[section]
\newtheorem{lemma}[theorem]{Lemma}
\newtheorem{prop}[theorem]{Proposition}
\newtheorem{cor}[theorem]{Corollary}

\newtheorem*{teoa}{Theorem A}
\newtheorem*{teoA}{Theorem A}
\newtheorem*{teoB}{Theorem B}

\newtheorem*{teoC}{Theorem C}

\theoremstyle{definition}
\newtheorem{definition}[theorem]{Definition}

\theoremstyle{remark}
\newtheorem{re}[theorem]{Remark}

\numberwithin{equation}{section}

\title{
Existence and Finiteness of equilibrium states for some Partially hyperbolic endomorphisms
}
\author[1]{Alexander Arbieto}
\author[2]{Eric Cabezas}
\affil[1,2]{Instituto de Matemática, Universidade Federal do Rio de Janeiro, RJ, Brazil.}
\affil[1]{\texttt{arbieto@im.ufrj.br}}
\affil[2]{\texttt{eric@im.ufrj.br}}

\date{}

\fancypagestyle{customstyle}{
	\fancyhf{}
	\fancyhead[LE,RO]{\thepage}
	\fancyhead[CE,CO]{\ifthenelse{\isodd{\thepage}}{Alexander Arbieto and Eric Cabezas}{Existence and Finiteness of equilibrium states Partially hyperbolic endomorphisms}}
	
}

\begin{document}
	\maketitle
	
	\setlength{\headheight}{14.49998pt}

\begin{abstract}
We establish the existence and finiteness of equilibrium states for a class of partially hyperbolic endomorphisms. In our first result, we assume that the central direction is simple. In the second result, we consider the case where there exists a dominated splitting along the central direction, which is decomposed into one-dimensional subbundles. This latter result extends the work of C. Alvarez and M. Cantarino \cite{alvarez2022existence} to higher-dimensional central directions. Finally, we demonstrate the finiteness of measures of maximal entropy under the assumption that the central direction is one-dimensional and the integrated Lyapunov exponent is bounded away from zero.
\end{abstract}

\section{Introduction}

In dynamical systems, there are different ways to quantify the amount of information generated by the considered map. One can approach this measurement from either a statistical or a topological perspective. An important tool that captures the complexity of the system is called \emph{entropy}, which essentially measures the exponential growth rate of expansion along orbits.

A classical result by P. Walters \cite{walters1975variational}, known as the \emph{Variational Principle}, establishes a deep connection between statistical and topological entropy, as well as their weighted counterpart, the pressure:

\begin{equation*}
    P(f, \phi) = \sup\left\{ h_{\mu}(f) + \int \phi \, d\mu : \mu \in \mathcal{M}(f) \right\} = \sup\left\{ h_{\mu}(f) + \int \phi \, d\mu : \mu \in \mathcal{M}_e(f) \right\},
\end{equation*}

where $ h_{\mu}(f) $ denotes the metric entropy of $ f $ with respect to the invariant measure $ \mu $. A measure attaining the supremum is called an \emph{equilibrium state}. In the special case where $\phi \equiv 0$, it is called a \emph{measure of maximal entropy} (abbreviated m.m.e.). 

However, the supremum is not always achieved; see, for instance, \cite{misiurewicz1973diffeomorphism}. To overcome this, R. Bowen \cite{bowen1972entropy} introduced the notion of \emph{entropy-expansiveness}, which ensures the upper semicontinuity of the metric entropy and, consequently, the existence of measures of maximal entropy. Anosov diffeomorphisms are classical examples of entropy-expansive systems. 

The existence of measure of maximal entropy has been established in several contexts, including group actions and systems with weaker forms of hyperbolicity, such as partially hyperbolic dynamics; see for example \cite{cowieson2005srb, Lorenzo2012entropy, climenhaga20, carrasco2021invariance, alvarez2021equilibrium}.

Despite significant progress, much remains to be explored in the context of partially hyperbolic endomorphisms. Recent advances include the classification of partially hyperbolic endomorphisms on surfaces in the presence of periodic center annuli by L. Hall and A. Hammerlindl \cite{hall2022classification}, as well as investigations into the stability of such systems on general closed manifolds by F. Micena and J. Costa \cite{micena2021some}.

In this work, we address the existence of measures of maximal entropy in the setting where the center bundle $E^c$ has dimension greater than one. Our first main result is the following:

\begin{teoA}
    Let $f: M \rightarrow M$ be a partially hyperbolic $C^1$ endomorphism with a simple center bundle. Then, for any continuous potential $\phi: M \to \mathbb{R}$, there exists an equilibrium state for $(f, \phi)$.
\end{teoA}

Informally, the assumption that $ E^c $ is simple means that it decomposes into one-dimensional subbundles, each of which integrates into the fibers of the inverse limit. This integrability of $ E^c $ allows us to employ techniques from \cite{carrasco2021invariance} to compute the entropy along the central curves.

Our second main result concerns systems with dominated splitting:

\begin{teoB}
    \label{teoB_1}
    Let $f: M \rightarrow M$ be a partially hyperbolic endomorphism with a dominated splitting
    \[
    TM = E^s \oplus E^1 \oplus \cdots \oplus E^{\ell} \oplus E^u,
    \]
    where $\dim(E^i) = 1$ for each $1 \leq i \leq \ell$. Then, for any continuous potential $\phi: M \to \mathbb{R}$, there exists an equilibrium state for $(f, \phi)$.
\end{teoB}

The same conclusion holds when considering a compact $f$-invariant subset instead of the entire manifold.

This extends the work of C. Alvarez and M. Cantarino \cite{alvarez2022existence}, who addressed the case where the center bundle is one-dimensional.

In \cite{lima2024measures}, the authors proved that $ f $ admits finitely many measures of maximal entropy for codimension-one partially hyperbolic endomorphisms, i.e., when the splitting is of the form $ E^c \oplus E^u $ with $ \dim(E^c) = 1 $. This result assumes that the topological entropy of $ f $ exceeds the logarithm of its topological degree, i.e., $ h_{\mathrm{top}}(f) > \log \deg(f) $.

Finally, under the assumption that the integrated Lyapunov exponent $ \lambda^c(\tilde{\mu}) $ is bounded away from zero and that the system admits a stable subbundle $ E^s $, we establish the following:

\begin{teoC}
    Let $ f: M \to M $ be a $ C^{1+\alpha} $ partially hyperbolic endomorphism (locally diffeomorphism) with $ \dim(E^c_{x_n}) = 1 $. If $ \lambda^c(\tilde{\mu}) \neq 0 $ for every measure of maximal entropy $ \tilde{\mu} $, then $ f $ admits finitely many ergodic measures of maximal entropy.
\end{teoC}

The approach based on homoclinic intersections first appeared in the work of J.~Buzzi, S.~Crovisier, and O.~Sarig~\cite{buzzi2022measures}, and was later adapted by Y.~Lima and M.~Poletti~\cite{lima2025homoclinic}, as well as by D.~Obata in the context of standard maps~\cite{obata2021uniqueness}. For the finiteness of measures of maximal entropy, there have been recent advances for diffeomorphisms, such as in the work of J.~Mongez and M.~Pacifico~\cite{mongez2024finite}, and J.~Mongez, M.~Pacifico, and M.~Poletti~\cite{mongez2025partially}, particularly addressing cases where the center subspace has dimension greater than one.

To prove this theorem, we begin by showing that the integrated central Lyapunov exponent $\lambda^c(\widetilde{\mu})$ depends continuously on the invariant measure. Assuming that there exists a sequence of ergodic measures of maximal entropy $\mu_n$; then, by compactness, a subsequence accumulates on a measure $\mu$. By \cite{misiurewicz1976topological}, the limit measure $\mu$ is also a measure of maximal entropy, which implies by the hypotheses that $\lambda^c(\widetilde{\mu}) \neq 0$.

The case $\lambda^c(\widetilde{\mu}) < 0$ follows directly from the argument in \cite{lima2024measures}.
We focus on the case $\lambda^c(\widetilde{\mu}) > 0$. In this setting, we construct a set exhibiting uniform expansion. Using a Pliss Lemma argument, we show that every ergodic measure of maximal entropy gives positive measure to this expanding set. Moreover, the unstable manifolds of points in this set have uniformly large length. Since the angle between stable and unstable manifolds is uniformly bounded away from zero, we apply an argument due to Buzzi, Crovisier, and Sarig—adapted for endomorphisms and proved in \cite{lima2024measures}—to conclude that all such measures of maximal entropy must accumulate on a single homoclinic class, leading to a contradiction.

\bigskip

This article is organized as follows. In Section~2, we prepare the ground by discussing the natural extension of the endomorphism and explaining how the smooth structure can be lifted in such a way that the invariant bundles remain well-defined and related. We also present some fundamental tools from the theory of partially hyperbolic diffeomorphisms that will be used throughout the paper. In Sections~3 and~4, we prove Theorems~A and~B, respectively, by establishing that $\tilde{f}$ is asymptotically $h$-expansive. Finally, in Section~5, we address the proof of uniqueness.

\bigskip

\textbf{Acknowledgements:} We are thankful to Marisa Cantarino and Davi Obata for helpful conversations on this work.
A. Arbieto was partially supported by CAPES, CNPq, Projeto Universal,  PRONEX-Dynamical Systems and FAPERJ E-26/201.181/2022 Programa Cientista do Nosso Estado from Brazil''. E. Cabezas was partially supported by CAPES.

\section{Preliminaries}

\subsection{Measures of maximal entropy}

It is well known the result of P. Walters \cite{walters1975variational} who relates the metric entropy and the topological entropy:

\begin{theorem}[Variational Principle]
\label{varprinc}

Let $f: X \to X$ be a continuous map of a compact metric space and $\varphi \in C^0(X)$, it holds

\begin{equation*}
    P(f, \phi) = \sup\{ h_{\mu}(f) + \int \phi d\mu : \mu \in \mathcal{M}(f ) \}  = \sup\{ h_{\mu}(f) + \int \phi d\mu : \mu \in \mathcal{M}_e(f ) \}.
\end{equation*}
    
\end{theorem}

This theorem tell us that  for a given potential $\phi$, we can approximate the topological pressure by its metric entropy with the asymptotical average of $\phi$.

%--------------------------------------------------------------------------
For completeness of the statement we will define the topological pressure and the metric entropy.

Let $f:X \to X$ be a continuous map in a compact metric space. We call a potential on $X$ to any continuous function $\phi: X \to \R$. For each $n \in \N$, define $\phi_n : X \to \R$ by $\phi_n = \sum_{i=0}^{n-1} \phi \circ f^i$. 

Given an open cover $\alpha$ of $X$, let 

\begin{equation*}
    P_n(f,\phi,\alpha)= \inf \left\{ \sum_{U \in \gamma } \sup_{ x \in U} e^{\phi_n(x)} : \gamma \hbox{ is a finite subcover of }\alpha^n \right\}.
\end{equation*}

This sequence $\log P_n(f, \phi,\alpha)$ is subadditive, and so the limit 

\begin{equation*}
    P(f,\phi,\alpha)= \lim_n \dfrac{1}{n} \log P_n(f,\phi,\alpha)
\end{equation*}

exists.  

Define the \textit{pressure} of the potential $\phi$ with respect to $f$ to be the limit $P(f,\phi)$ of $P(f,\phi,\alpha)$ when the diameter of $\alpha$ goes to zero. In the particular case that $\phi \equiv 0$, $P(f,\phi)$ defines the topological entropy.  In the particular case where we restrict the definition of entropy to a compact set $K \subset X$, we refer to it as the \emph{conditional entropy} and denote it by $h(f, K)$.

Let $\mathcal{P}$ be a finite partition of $X$, with finite entropy, this means

\begin{equation*}
    H_{\mu}(\Pa) := \sum_{P \in \Pa} -\mu(P) \log \mu (P) < \infty.
\end{equation*}

Set

\begin{equation*}
    \Pa^n = \bigvee_{i=0}^{n-1} f^{-i}(\Pa) \hbox{ for each } n \geq 1,
\end{equation*}

where $\mathcal{A} \vee \mathcal{B} := \{ A \cap B : A \in \mathcal{A}, B \in \mathcal{B} \}$ for $\mathcal{A} $ and $ \mathcal{B}$ partitions of $X$. It is clear that the sequence $\Pa^n$ is non-decreasing, that is, $\Pa^n \prec \Pa^{n+1}$, this means that $\Pa^n$ is refined by $\Pa^{n+1}$. Furthermore,  $H_{\mu}(\Pa)$ is subadditive. It follows that the limit 

\begin{equation*}
    h_{\mu}(f,\Pa)= \lim_{n} \dfrac{1}{n} H_{\mu}(\Pa^n),
\end{equation*}

exists. We call $ h_{\mu}(f,\Pa)$ the \textit{entropy of f with respect to the partition $\Pa$}. Observe that this entropy is a non-decreasing function of the partition.

Finally, the \textit{entropy} of the system $(f, \mu)$ is defined by

\begin{equation*}
    h_{\mu}(f) = \sup_{\Pa} h_{\mu}(f,\Pa).
\end{equation*}

For further discussion, see, for example, \cite{viana2016foundations}. In the case that there exists a measure who attains the supremum, it is called an \textit{equilibrium state}, in particular when the potential $\varphi$ constant equals to zero, it is called a \textit{measure of maximal entropy}.

\bigskip
In the case that the system is expansive, $f$ is said to be expansive if there is a constant $\epsilon_0>0$ called the expansivity constant, such that for any pair of points $x \neq y$ in $X$ there is an integer $n$ such that $d(f^n(x),f^n(y)) \geq \epsilon_0$. It is well known that the metric entropy is upper semicontinuous. This basically follows from the fact that each partition with diameter less than the constant of expansivity $\epsilon_0$ will dinamically generate the Borel $\sigma$-algebra, and for any measure $\mu$ we can find a partition $\mathcal{P}$ with $\mu(\partial \mathcal{P})=0$, with diameter less than $\epsilon_0$. R. Bowen in \cite{bowen1972entropy} developed a condition to get upper semicontinuity of the metric entropy, which is now called entropy expansiveness.
%DEFINIR EXPANSIVE
%Dynamically generate

%------------------------------------------------------------
To introduce the concept of entropy expansiveness ($h$-expansiveness), we begin by recalling the definition of dynamical balls. Let $ f: X \to X $ be a homeomorphism.

For $n \in \N$, the $(n, \epsilon)$-dynamical ball centered at $x \in X$ is defined as:
\begin{equation*}
    B(x, n, \epsilon) := \{ y \in X : d_n(x, y) < \epsilon \},
\end{equation*}
where 
\[
    d_n(x, y) := \max \{ d(f^j(x), f^j(y)) : j \in \{0, 1, \dots, n\} \}.
\]
Additionally, the dynamical ball of radius $\epsilon > 0$ centered at $x \in X$ is given by:

\[
    \Gamma_{\epsilon}(x) = \{ y \in X : d(f^n(x), f^n(y)) < \epsilon, \, n \in \mathbb{Z} \}.
\]
In the special case where $f$ is expansive, there exists an $\epsilon > 0$ such that $\Gamma_{\epsilon}(x) = \{x\}$ for all $x \in X$. This property reflects the fact that distinct points eventually separate under the action of $f$ when expansivity holds.

%------------------------------------------------------------
To generalize the notion of expansivity, we introduce the following construction. Define:

\begin{equation*}
    \Phi_{\epsilon}(x) = \bigcap_{n=0}^{\infty} B(x, n, \epsilon),
\end{equation*}
where $\Phi_{\epsilon}(x)$ represents the set of points whose forward orbits remain $\epsilon$-close to the orbit of $x$ for all future iterations.

\begin{definition}
Let $f: M \to M$ be a continuous map. We define the following notions:

- The map $f$ is said to be $h$-\textit{expansive} if $\exists \epsilon >0$ such that
        \[
        h^*_f(\epsilon) := \sup_{x \in M} h(f, \Phi_{\epsilon}(x)) = 0,
        \]
        where $h(f, \Phi_{\epsilon}(x))$ denotes the topological entropy restricted to the set $\phi_{\epsilon}(x)$.

- The map $f$ is called \textit{asymptotically $h$-expansive} if 
        \[
        \lim_{\epsilon \to 0} h^*_f(\epsilon) = 0.
        \]
    
\end{definition}

The \textit{tail entropy} of $f$, introduced by M. Misiurewicz in \cite{misiurewicz1976topological}, is defined as:
\[
h^*(f) := \lim_{\epsilon \to 0^+} h^*_f(\epsilon).
\]
It is worth noting that $h$-expansiveness implies asymptotic $h$-expansiveness. 

When $f$ is a homeomorphism, it holds that $\Gamma_{\epsilon}(x) \subset \Phi_{\epsilon}(x)$. Furthermore, R. Bowen \cite{bowen1972entropy} established the equality:
\[
h^*_f(\epsilon) = \sup_{x \in M} h(f, \Phi_{\epsilon}(x)) = \sup_{x \in M} h(f, \Gamma_{\epsilon}(x)),
\]
where $\Gamma_{\epsilon}(x)$ denotes the set of points whose orbits remain $\epsilon$-close to the orbit of $x$ for all time (both forward and backward).
%------------------------------------------------------------

\begin{theorem}{\cite{misiurewicz1973diffeomorphism}}
\label{semicont}
    If $f: M \to M$ is asymptotically $h$-expansive, then the entropy function is upper semi-continuous. In particular, $f$ admits a measure of maximal entropy.  
\end{theorem}
%\textcolor{RedViolet}{
\begin{re}
    If $h_{\mu}(f)$ is upper semi continuous, there is  an equilibrium state for any continuous potential, since $\int \varphi d\mu$ is continuous. 
\end{re}
%}
\subsection{Inverse Limit}

In the context of endomorphisms, the non-uniqueness of past trajectories for points in the metric space can pose significant challenges when attempting to generalize results from diffeomorphisms. This motivates the construction of an extended system in which the dynamics are represented by a homeomorphism, thereby ensuring a well-defined inverse dynamics. As we will see in this subsection, a suitably defined lift to such an extended system will establish a correspondence between invariant measures and entropy between the original and the extended systems.

%------------------------------------------------------------------------
Let $(X, d)$ be a metric space, and let $f: X \to X$ be a continuous map. We define the inverse limit (also referred to as the natural extension) of the dynamical system as follows:

- The space $\widetilde{X}$ is defined as:
    \[
        \widetilde{X} := \{ \widetilde{x} = (x_n)_{n \in \mathbb{Z}} \in X^{\mathbb{Z}} : x_{n+1} = f(x_n), \, n \in \mathbb{Z} \}.
    \]

- The map $\widetilde{f}: \widetilde{X} \to \widetilde{X}$ is given by:
    \[
        \widetilde{f}((x_n)_n) = (x_{n+1})_n.
    \]

- A metric $\widetilde{d}$ on $\widetilde{X}$ is defined as:
    \[
        \widetilde{d}(\widetilde{x}, \widetilde{y}) = \sum_{k=-\infty}^\infty \frac{d(x_k, y_k)}{2^{|k|}},
    \]
    where $\widetilde{x} = (x_n)_n$ and $\widetilde{y} = (y_n)_n$ are elements of $\widetilde{X}$.

The natural projection $p: \widetilde{X} \to X$ is defined by:
\[
    p((x_n)_n) = x_0.
\]
This projection guarantees that the following diagram commutes:
\[
\begin{tikzcd}
\widetilde{X} \arrow{r}{\widetilde{f}} \arrow[swap]{d}{p} & \widetilde{X} \arrow{d}{p}\\
X \arrow{r}{f} & X
\end{tikzcd}
\]

In the case where the endomorphism $f$ is a local diffeomorphism on a manifold $X$, the lift $\widetilde{X}$ forms a solenoidal manifold, which is locally homeomorphic to the product of a disc and a Cantor set. For further details on this structure, we refer the reader to \cite{verjovsky22} and (\cite{andersson2022non}, Appendix A).

\textbf{Notation}: We denote $B(x,\delta) \subset X$ to a neighborhood of $x$ with radius $\delta$ and $B(\tilde{x},\delta) \subset \widetilde{X}$ to a neighborhood of $x$ with radius $\delta$ such that $p(\tilde{x})=x$.

For any $f$-invariant measure $\mu$, it is possible to construct a lift of this measure via the natural projection, resulting in an $\widetilde{f}$-invariant measure $\widetilde{\mu}$. The following proposition establishes that this correspondence is in fact bijective.

\begin{prop}
\label{bijm}
Let $(X,d)$ be a compact metric space and $f: X \to X$ continuous. For any $f$-invariant probability measure $\mu$ on $X$, there is a unique $\widetilde{f}$-invariant probability measure $\widetilde{\mu}$ on $\widetilde{X}$ such that $p_* \widetilde{\mu}= \mu$.     
\end{prop}
Furthermore, it is well known that the inverse limit construction preserves entropy. Specifically, we have the following result:

\begin{prop}
\label{meas}
    Let $T:X \to X$ be a continuous map  on the compact metric space $X$ with an invariant Borel probability measure $\mu$. Let $\widetilde{X}$ be the inverse limit space of $(X,T)$, with $\widetilde{\mu}$ an $\widetilde{T}$-invariant Borel probability measure on $\widetilde{X}$ such that $p_* \widetilde{\mu} = \mu $. Then

    \begin{equation*}
        h_{\mu}(T)= h_{\widetilde{\mu}}(\widetilde{T})
    \end{equation*}
    where $h_\mu(T)$ and $h_{\widetilde{\mu}}(\widetilde{T})$ denote the metric entropies of $T$ and $\widetilde{T}$ with respect to $\mu$ and $\widetilde{\mu}$, respectively.
    
\end{prop}

By Theorem \ref{varprinc}, topological entropy of $(\widetilde{X}, \widetilde{f})$ coincides with that of $(X, f)$.

%------------------------------------------------------------------------
We refer the reader to (\cite{qian2009smooth}, Proposition I.3.1 and Proposition I.3.4) for the proof of this proposition.

%-------------------------------------------------------------------------------

\subsection{Partially hyperbolic endomorphism and dominated splitting}

Since the natural candidate to be the unstable subbundle $E^u$ is not well defined, for it depends on the past orbit of the point, the definition of partial hyperbolicity for non-invertible maps is not direct. A solution for that is to define the tangent bundle and the subspaces along an orbit $\widetilde{x}$, where $\widetilde{x} = (x_n)_{n \in \Z} \in p^{-1}(x)$.

We can associated a vector bundle $T \widetilde{M}$ to $\widetilde{M}$ by the pullback under $p$ from $TM$, where $(\widetilde{x},v) \in T \widetilde{M}$, $\widetilde{x}$ belongs to $\widetilde{M}$  and $v$ is a tangent vector in $ T_{x_0}M$. The derivative $Df$ lifts to a map $D \widetilde{f}$ of $T \widetilde{M}$ in a natural way.

\begin{definition}
    Let $f: M \to M$ be a $C^1$ local diffeomorphism, we say tat $f$ is a  \textit{partially hyperbolic endomorphism} if for any $(x_n)_{n \in \mathbb{Z}} \in \widetilde{M}$, there exists a  splitting $T_{x_n}M = E^u_{x_n} \oplus E^c_{x_n} \oplus E^s_{x_n}$ satisfying 

\begin{enumerate}
    \item $Df_{x_n}(E^{\sigma}_{x_n})= E^{\sigma}_{f(x_n)}$, $ \sigma \in \{ s,c,u \}$, for any $n \in \mathbb{Z}$ ($Df$-invariant),

    \item If $\widetilde{x} \in \widetilde{M}$, then

\begin{equation*}
    ||Df_{x_i}v^s||< ||Df_{x_i}v^c|| < ||Df_{x_i} v^u||
\end{equation*}

    for  $v^{\sigma} \in E^{\sigma}(x_i)$ unitary vectors, where $\sigma \in \{ s,c,u \}$.  Furthermore $||Df|_{E^s_{x_i}}|| < 1$, $||Df|_{E^u_{x_i}}|| > 1$ for any $\widetilde{x} \in \widetilde{M}$ and $i \in \mathbb{Z}$.    
\end{enumerate}
    
\end{definition}
 
%---------------------------------------------------------------------------------

We say that $f$ is \textit{ dynamically coherent} if $E^c$ integrates for any given $\widetilde{x}$, this means that there exist a unique $C^1$ manifold in $M$ which depends on $(x_n)_n$ and by definition of $D\tilde{f}$ is tangent to $E^c_{\tilde{x}}$. for the setting of our first result we precise a desintegration along each subspace of the central space $E^c$, for this we define: 

\begin{definition}
    Let $f: M \to M$ be a partially hyperbolic endomorphism. Let $\widetilde{x} \in \widetilde{M}$, we say that its center bundle $E^c_{\widetilde{x}}$ is \textit{simple} if 

\begin{itemize}
    \item $E^c(x_n) = E^1(x_n) \oplus \cdots \oplus E^{\ell}(x_n)$ with $\dim E^i(x_n)=1 $, for every $i =1,...,\ell$, $\widetilde{x}= (x_n)_{n} \in \widetilde{M}$,
    \item For every $S \subset \{1, \cdots , \ell\}$ the bundle $E^S:= \oplus_{i \in S} E^i $ integrates to an $f$-invariant foliation $\mathcal{F}^S$ along each fiber of $\widetilde{x}$ (in particular $E^c = E^{\{1, \cdots , \ell \}}$ is integrable). Consequently, if $S \subset S'$ then $\mathcal{F}^S$ sub-foliates $\mathcal{F}^{S'}$.
\end{itemize}
    
\end{definition}

Analogously to Anosov endomorphism \cite{mane1975stability} and partially hyperbolic endomorphism \cite{micena2021some}, it seems natural to propose the following definition for partially hyperbolic endomorphism with dominated splitting:

\begin{definition}
    Given $f: M \to M$ a partially hyperbolic endomorphism, we say that the splitting $T_{x_n} M = E^s(x_n) \oplus E^1(x_n) \oplus \cdots \oplus E^{\ell}(x_n) \oplus E^u(x_n)$ is \textit{dominated} if there exist $C>0$ $\lambda \in (0,1)$ such that 

\begin{equation}
\label{ds}
    ||Df^n|_{E^{cs,i}(x_n)}|| < C \lambda^n ||Df^n|_{E^{cu,i+1}(x_n)}||,
\end{equation}
for all $\widetilde{x} \in \widetilde{M}, i \in \{0,...,k\}$ and $n \geq 0$.   
Where
\begin{eqnarray}
    \label{csu}
    E^{cs,i}(x_n):= E^s(x_n) \oplus E_1(x_n) \oplus \cdots \oplus E_i(x_n) \nonumber,\\ 
    E^{cu,i+1}(x_n):=E_{i+1}(x_n) \oplus \cdots \oplus E_k(x_n) \oplus E_u(x_n).
\end{eqnarray}
We also let $E^{cs,0}=E^s$ and $E^{cu,k+1}=E^u$ and write $s=dim(E^s)$ and $u=dim(E^u)$.
 
\end{definition}
%%%%%%%%%%%%%%%%%%%%%%%%%%%%%%%%%%%%%%%%%%
The definition of this splitting is unique, continuous along $\widetilde{M}$ and the  angles between $E^{cs,i}$ and $E^{cu,i+1}$ are bounded away from zero, the same arguments from the diffeomorphism case works for this splitting, see for example \cite{crovisier2015introduction}.

It is not difficult to choose an adapted Riemannian metric equivalent to the original such that $C=1$, see for instance \cite{micena2021some}. 

Since $f: M \to M$ is a local diffeomorphism, for any $\tilde{x}=(x_n)_n$ we can define $Df^{-1}x_{-n} : T_{x_{-n+1}}M \to T_{x_{-n}}M$, where $\tilde{f}^{-1}(x_{-n+1})=x_{-n}$.

\subsection{Lyapunov exponents}

The following theorem is an endomorphism version of Oseledets  theorem, given in \cite{qian2009smooth} 

\begin{theorem}
    Let $f$ be a $C^1$ endomorphism on $M$. there exists a Borel set $\tilde{\Delta} \subset \widetilde{M}$ such that $\tilde{f}(\tilde{\Delta}) = \tilde{\Delta}$ and $\tilde{\mu}(\tilde{\Delta})=1$, where $\tilde{\mu}:= p *\mu$. Furthermore for every $\tilde{x}= \{ x_n\}_{n \in \Z} \in \tilde{\Delta}$, there is a splitting of the tangent space $T_{x_0}M$

\begin{equation*}
    T_{x_0}M = E_1(\tilde{x}) \oplus E_2(\tilde{x}) \oplus \cdots \oplus E_{r(\tilde{x})}(\tilde{x})
\end{equation*}

and numbers $+\infty > \lambda_1(\tilde{x}) > \lambda_2(\tilde{x}) > \cdots  > \lambda_{r(\tilde{x})}(\tilde{x})> -\infty$ and $m_i(\tilde{x})$ for $i = 1,2,...,r(\tilde{x})$, such that 

\begin{enumerate}
    \item $Df_{x_n}: T_{x_{n}}M \to T_{x_{n+1}}M$ is an isomorphism, $\forall n \in \Z$,
    \item $r(\cdot), \lambda_i(\cdot)$ and $m_i(\cdot)$ are $\tilde{f}$-invariant i.e.,

\begin{equation*}
    r(\tilde{f}(\tilde{x})) = r(\tilde{x}), \; \: \lambda_i(\tilde{f}(\tilde{x})) = \lambda_i(\tilde{x}), \: \hbox{ and } \: m_i(\tilde{f}(\tilde{x})) = m_i(\tilde{x})
\end{equation*}

for each $i = 1,...,r(\tilde{x})$.

\item $dim \, E_i(\tilde{x}) = m_i(\tilde{x})$ for all $n \in \Z$ and $1 \leq i \leq r(\tilde{x})$.  
\item The splitting is invariant under $Df$, i.e.,

\begin{equation*}
    Df_{x_n} E_i(\tilde{f}^n(\tilde{x})) = E_i(\tilde{f}^{n+1}(\tilde{x}))
\end{equation*}
for all $n \in \Z$ and $1 \leq i \leq r(\tilde{x})$.
\item For any $n,m \in \Z$, let

\begin{equation*}
 Df^m_n(\tilde{x}) = \begin{cases}
       Df^m_{x_n} & \text{if } m >0 \\
       id & \text{if } m=0 \\
       (Df_{n+m}^{-m})^{-1} & \text{if } m<0
     \end{cases}
\end{equation*}

then,

\begin{equation*}
    \lim_{m \to \pm \infty} \dfrac{1}{m} \log |Df^m_n(\tilde{x}) \xi| = \lambda_i (\tilde{x}),
\end{equation*}

for all $0 \neq \xi \in E_i(\tilde{f}^n(\tilde{x}))$, $1 \leq i \leq r(\tilde{x})$.

\end{enumerate}
    
\end{theorem}

    The numbers $\{ \lambda_i(\tilde{x})\}_{i=1}^{r(\tilde{x})}$ given in the last theorem are called \textit{the Lyapunov exponents of} $(\widetilde{M},\tilde{f}, \tilde{\mu})$ at $\tilde{x}$, and $m_i(\tilde{x})$ is called the multiplicity of $\lambda_i(\tilde{x})$.

\begin{re}

The splitting provided by Oseledets' theorem is compatible with the dominated splitting. That is, given a dominated splitting, each subbundle is contained within an Oseledets subspace.

\end{re}

The integrated Lyapunov exponent $\lambda_i(\mu)$ is defined as follows:

\begin{equation*}
    \lambda_i(\tilde{\mu}) := \int \lambda_i(\tilde{x}) d \tilde{\mu}(\tilde{x}) 
\end{equation*}

\begin{re}
In the case that $\mu$ is ergodic, then $\lambda_i(\mu) = \lambda_i(x)$ for $\mu$- almost everywhere.
\end{re}

\subsection{Fake foliations}

We say that $f: M \to M$ a partially hyperbolic diffeomorphism is dynamical coherent if $E^{cs}$ and $E^{cu}$ are integrable, this property does not always hold even when the center direction is unidimensional, see for example \cite{RHRHU16}, but we have locally a fake foliation proved in \cite{BurnsWilkinson10}, and we will use these fake foliations in each path component of $\widetilde{M}$. 

\begin{re}

 Burns and Wilkinson developed the fake foliations to work locally in the partially hyperbolic setting. Fake foliations are a family of sets that are similar to the foliations of a neighborhood, in the sense that they are projections of the exponential map of sub spaces that belong in the same cone of the respective sub space (they are explained with more detail in the proposition below).

 From the previous sections, given a neighborhood  $B(x,\epsilon)$ of $x$ in $M$,  in general we are unable to use the results from diffeomorphism in $p^{-1}(B(x,\epsilon))$, because of the structure of $p^{-1}(x)$ ( a Cantor set) but if we restrict our analysis to a fiber, that we will call $\widetilde{B}(\widetilde{x}, \epsilon)$ such that is isomorphic to $p^{-1}(B(x,\epsilon)) \times \{ \widetilde{x} \}$, locally we can apply all the arguments known in each fiber.
\end{re}

\begin{prop}
\label{fakef}
    Let $f: M \to M$ be a $C^1$ local diffeomorphism and $\Lambda$ a compact $f$-invariant set with partially hyperbolic splitting
\begin{equation*}
    T_{\widetilde{\Lambda}}M = E^s \oplus E_1 \oplus \cdots \oplus E_k \oplus E^u.
\end{equation*}
Let $E^{cs,i}$ and $E^{cu,i}$ be as in equation (\ref{csu}) and consider \textcolor{blue}{an} extension\footnote{It is always possible to find an extension restricted to a cone field. Therefore, it is not unique.} $\overline{E}^{cs,i}$ and $\overline{E}^{cu,i}$ to a small neighborhood of $\widetilde{\Lambda}$.\\
Then for any $\epsilon>0$ there exist constants $R>r>r_1>0$ such that, for every $\widetilde{p} \in \widetilde{\Lambda}$, each fiber $\widetilde{B}(\widetilde{p},r)$ is foliated by foliations $\widehat{W}^u(\widehat{p}), \widehat{W}^s(\widetilde{p}), \widehat{W}^{cs,i}(\widetilde{p})$ and $\widehat{W}^{cu,i}(\widetilde{p})$, $i \in \{1,...,k \}$, such that for each $\beta \in {u,s,(cs,i),(cu,i)}$ the following properties hold:

\begin{enumerate}
    \item Almost tangency of the invariant distributions. For each $\widetilde{q} \in \widetilde{B}(\widetilde{p},r)$, the leaf $\widehat{W}_{\widetilde{p}}^{\beta}(\widetilde{q})$ is $C^1$, and the tangent space $T_{\widetilde{q}} \widehat{W}^{\beta}_{\widetilde{p}}(\widetilde{q})$ lies in a cone of radius $\epsilon$ about $E^{\beta}(q_0)$.
    \item Coherence. $\widehat{W}^s_p$ subfoliates $\widehat{W}^{cs,i}_{\widetilde{p}}$ and $\widehat{W}^u_p$ subfoliates $\widehat{W}^{cu,i}_{\widetilde{p}}$ for each $i \in \{1,...,k\}$.
    \item Local invariance. For each $\widetilde{q} \in \widetilde{B}(\widetilde{p},r_1)$ we have
    \begin{equation*}
        \widetilde{f}(\widehat{W}^{\beta}_{\widetilde{p}}(\widetilde{q},r_1)) \subset \widehat{W}^{\beta}_{\widetilde{f}(\widetilde{p})}(\widetilde{f}(\widetilde{q})) \hbox{ and } \widetilde{f}^{-1}(\widehat{W}^{\beta}_{\widetilde{p}}(\widetilde{q},r_1)) \subset \widehat{W}^{\beta}_{\widetilde{f}^{-1}(\widetilde{p})}(\widetilde{f}^{-1}(\widetilde{q})),
    \end{equation*}
     here $\widehat{W}^{\beta}_{\widetilde{p}}(\widetilde{q},r_1)$ is the connected component of $\widehat{W}^{\beta}_{\widetilde{p}}(\widetilde{q}) \cap \widetilde{B}(\widetilde{q},r_1)$ containing $\widetilde{q}$.
     \item Uniqueness. $\widehat{W}^s_{\widetilde{p}}(\widetilde{p})=\widetilde{W}^s_{\widetilde{p}}(\widetilde{p},r)$ and  $\widehat{W}^u_{\widetilde{p}}(\widetilde{p})=\widetilde{W}^u_{\widetilde{p}}(\widetilde{p},r)$.
\end{enumerate}

\end{prop}

\begin{re}
    The same arguments of Burns, Wilkinson works for locally diffeomorphism, but in this case we restrict our analysis to a fixed past, this is a preimage $(x_n)_n \in p^{-1}(x)$ of $x$ in the inverse limit, where $Df_{x_n}: E^i_{x_n} \to E^i_{x_{n+1}}$. Using this linear isomorphism, we can apply the graph transform as in \cite{qian2009smooth}, for unstable manifolds in the non uniformly hyperbolic endomorphism case, to obtain the fixed points as in Burns, Wilkinson argument. We make the constructions in $M$, but to obtain the similar behavior as in the diffeomorphism case we look at the fibers of $p^{-1}(B(x,r))$. In this case each fiber of $p^{-1}(B(x,r))$ will be foliated by these fake foliations, where each fiber is isomorphic to $B(x,r)$ \footnote{During the submission of this work, we became aware of a recent preprint \cite{hammerlindl2025ergodicity} by A.~Hammerlindl and A.~Taylor, in which the authors also develop a version of fake foliations for endomorphisms.}.
\end{re}

By choosing the neighborhood $V(\widetilde{\Lambda})$ of an $\widetilde{f}$-invariant set $\widetilde{\Lambda}$ above sufficiently small we have that $\Lambda_V$ also satisfies the hypotheses in Proposition \ref{fakef} and hence the points in $\widetilde{\Lambda}_V$ have fake foliations as in the proposition.

\begin{re}
\label{ffol2}
    For $\epsilon$ sufficiently small, the transversality of the invariant bundles for $\Lambda_V$ implies that, for all $\widetilde{p} \in \Lambda_V$ and every $\widetilde{x}$ and $\widetilde{y}$ sufficiently close to $\widetilde{\Lambda}_V$, then $\widehat{W}^{cs,i}_{\widetilde{p}}(\widetilde{x}) \cap \widehat{W}^{cu,i+1}_{\widetilde{p}}(\widetilde{y})$ consists of a single point for all $i \in \{0,...,k\}$. Here $\widehat{W}^{s}_{\widetilde{p}}(\widetilde{x})=\widehat{W}^{cs,0}(\widetilde{x})$ and $\widehat{W}^{u}_{\widetilde{p}}(\widetilde{y})=\widehat{W}^{cu,k+1}(\widetilde{y})$
\end{re}

\begin{figure}[h]
	\centering
	\includegraphics[scale=0.38]{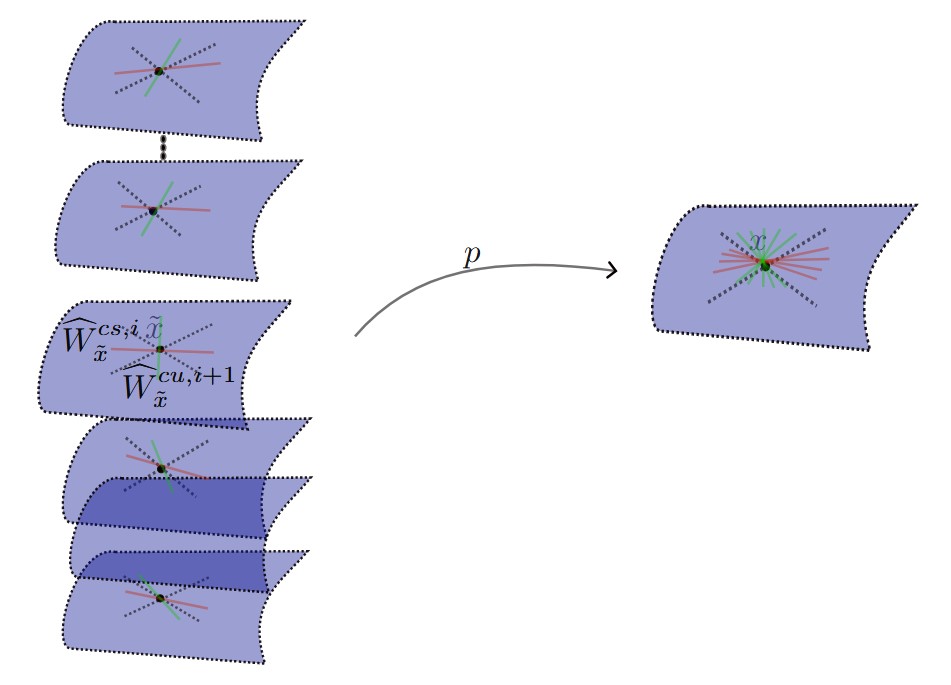}
	\caption{Fake foliations of $E^{cs,i}_{\widetilde{x}}$ and $E^{cu,i+1}_{\widetilde{x}}$}
    \label{fig:theorem1-3}
\end{figure}

\section{Proof of Theorem A}

%------------------------------------------------------------------------
\begin{theorem}
\label{theo_hexp1}
    Let $f: M \rightarrow M$ be a partially hyperbolic endomorphism with simple center direction $E^c$, then $\widetilde{f}$ is asymptotically $h$-expansive.
\end{theorem}

The proof is divided into three parts. First, we demonstrate that the dynamical balls are contained within $\widetilde{D}^{cs}_{\alpha}(\widetilde{x})$, a disc in the stable manifold $\widetilde{W}^s(\widetilde{x})$ foliated by center manifolds. Next, we employ the arguments from \cite{carrasco2021invariance} to show that the entropy along each center manifold is zero. Finally, we conclude that the entire disc $\widetilde{D}^{cs}_{\alpha}(\widetilde{x})$ has zero entropy, which concludes the proof.

The first part of the proof follows a similar strategy to that in \cite{alvarez2022existence}. However, in our setting, we make use of the fact that the center bundle $E^c$ is simple, which ensures its integrability at every point $\widetilde{x}$.

For any closed manifold $M$, their extension on the inverse limit is a fiber bundle $(\widetilde{M},M,p,C)$, where $C$ is a cantor set, see \cite{aoki1994topological} for further discussion. Choose $\beta > 0$ sufficient small such that $p^{-1}(B(x, \beta )) \simeq B(x, \beta ) \times p^{-1}(x) $. The symbol $\simeq$ means that the sets are homeomorphic. There exist $\delta \in (0, \beta)$ such that for any $x \in M$, $\widetilde{B}(\widetilde{x},\beta) \simeq B(x,\beta) \times \{ \widetilde{x}\}$, the connected component of $p^{-1}(B(x,\beta))$ which contains $\widetilde{x}$. If $\widetilde{y} \in \widetilde{B}(\widetilde{x},\beta)$  such that $\tilde{d}(\widetilde{x},\widetilde{y})< \delta$, then:

\begin{itemize}
    \item $\widetilde{W}^s_{\beta}(\widetilde{x}) \cap \widetilde{W}^c_{\beta}(\widetilde{y})$ is a unique point.
    \item $\widetilde{W}^u_{\beta}(\widetilde{x}) \pitchfork  \widetilde{D}^{cs}_{\beta}(\widetilde{y})$ is a unique point, where
%---------------------------------------------------------------
%----------------------------------------------------------------------------------------------
$$\widetilde{D}^{cs}(\widetilde{y}) = \bigcup_{w \in \widetilde{W}^c_{\delta}(\widetilde{y})}\widetilde{W}^s(w)$$

  is a disc tangent to $E^c \oplus E^s$.
\end{itemize}

This follows by the continuity of the foliations. Let $\alpha>0$ such that $\lambda \cdot \alpha < \delta < \beta$, where $\lambda= \max_{x \in M} ||Df_x|_{E^u}||$. Now, we can define $V(\widetilde{x}) \subset B(x, \beta)$, with diameter $\alpha$.

\begin{equation*}
    V(\widetilde{x}) = \bigcup_{w \in W^{cs}_{\alpha}(\widetilde{x})} p ( \widetilde{W}^u_{\alpha}(w))
\end{equation*}
%----------------------------------------------------------------------------------------------

%----------------------------------------------------------------------------------------------
We can lift $V(\widetilde{x})$ to $\widetilde{V}(\widetilde{x}) := p^{-1}(V(\widetilde{x})) \cap \widetilde{B}(\widetilde{x}, \beta)$ to the fiber in $\widetilde{x}$. 
%Note that $\widetilde{V}(\widetilde{x})$ have the LPS and the image by $\widetilde{f}$, have the same property as well, by the definition of $\alpha$.
There exists $\epsilon_0 > 0$ such that for any $0 < \epsilon \leq \epsilon_0$, we have:
\[
B(\widetilde{x}, \epsilon) \cap \widetilde{B}(\widetilde{x}, \beta) \subset \widetilde{V}(\widetilde{x}).
\]
This follows from the fact that $\widetilde{V}(\widetilde{x})$ is uniformly bounded, with $\alpha$ independent of $x$. Here, $B(\widetilde{x}, \epsilon)$ denotes the ball centered at $\widetilde{x}$ in the inverse limit space. For simplicity of notation we write $\widetilde{x}^n$ instead of $\widetilde{f}^n(\widetilde{x})$. In order to get  $h(f,\Gamma_{\epsilon}(\widetilde{x}))=0$ it suffices to prove that $h(f, \widetilde{D}^{cs}_{\alpha}(\widetilde{x})) =0$. Since $\Gamma_{\epsilon}(\widetilde{x}) \subset \widetilde{D}^{cs}_{\alpha}(\widetilde{x})$, by the following lemma.

\begin{lemma}
\label{lemma_inc}
    $\Gamma_{\epsilon} (\widetilde{x})  \subset \widetilde{D}^{cs}_{\alpha}(\widetilde{x})$.
\end{lemma} 

\begin{proof}
    Let $\widetilde{z} \in \Gamma_{\epsilon}$, where $\widetilde{z}^n  \in \widetilde{B}(\widetilde{x}^n, \epsilon)$, then $\widetilde{z}^n \in \widetilde{V}(\widetilde{x}^n)$. 
%\widetilde{z}^n := \widetilde{f}^n(\widetilde{z})
Consider the projections: 

\begin{equation*}
    \begin{aligned}
        p^{cs}&:& \widetilde{V}(\widetilde{x}) & \rightarrow & \widetilde{D}^{cs}_{\alpha}(\widetilde{x})\\
        & & \widetilde{z} & \rightarrow & \widetilde{W}^{u}_{\alpha}(\widetilde{z}) \cap \widetilde{D}^{cs}_{\alpha}(\widetilde{x})
    \end{aligned}
\end{equation*}

and

\begin{equation*}
    \begin{aligned}
        p^{c}&:& \widetilde{D}^{cs}(\widetilde{x}) & \rightarrow & \widetilde{W}^{c}_{\alpha}(\widetilde{x})\\
        & & \widetilde{y} & \rightarrow & \widetilde{W}^{s}_{\alpha}(\widetilde{y}) \cap \widetilde{W}^{c}_{\alpha}(\widetilde{x})
    \end{aligned}
\end{equation*}
Set $\widetilde{y}^n := p^{cs}(\widetilde{z}^n)$ and $\widetilde{w}^n := p^c(\widetilde{y}^n)$.

If $\widetilde{z}^n \not \in \widetilde{D}^{cs}(\widetilde{x}^n)$ for some $n \in \mathbb{N}$, then  $\widetilde{z}^n \neq \widetilde{y}^n$. Since $\widetilde{y}^n \in \widetilde{W}^u(\widetilde{x}^n)$,there exist $m \in \mathbb{N}$  such that  $d(\widetilde{y}^{n+m}, \widetilde{z}^{n+m}) > \alpha$. This implies $\widetilde{z}^{n+m} \not \in \widetilde{V}(\widetilde{x}^{n+m})$, which is a contradiction. 
\end{proof}
%-----------------------------------------------------------------------------------------------
The following definitions are motivated by \cite{carrasco2021invariance}.

\begin{definition}
    \textcolor{blue}{An} $\ell$-dimensional compact set $G \subset \widetilde{M}$ has \textit{geometrical wiring} if there exist $1$-dimensional foliations $W^1,...,W^{\ell}$ by compact leaves and $\widetilde{x} \in G$ such that

\begin{equation*}
    D_1= W^1(\widetilde{x}), \:\; D_{i+1}= \bigcup_{\widetilde{y} \in D_i}W^i(\widetilde{y}), 1\leq i \leq \ell-1.
\end{equation*}

It holds
\begin{enumerate}
    \item $D_{\ell}=G$;
    \item for every i, if $y \in D_i$, then $W^{i+1}(y) \cap D_i = \{ y \}$.
\end{enumerate}
    
\end{definition}

The second condition naturally defines a projection along $D_i$, $\pi^G_i: D_{i+1} \to D_i$. Assuming that $\widetilde{f}^n(G)$ is a geometrical wiring for every $n \geq 0$, we say that $W^1_{\widetilde{f}^n(G)},...,W^{\ell}_{\widetilde{f}^n(G)}$ are $\widetilde{f}$-invariant if 

\begin{equation*}
    \widetilde{f}(W^i_{\widetilde{f}^n(G)}(\widetilde{x})) =W^i_{\widetilde{f}^{n+1}(G)} (\widetilde{f}(\widetilde{x})) , \:\; \hbox{ for every } \widetilde{x} \in G, i=1,...,\ell, n \in \mathbb{N}.  
\end{equation*}

This induces well defined $\widetilde{f}$-invariant foliations on 

\begin{equation*}
    G_{\infty}:= \bigcup_{n \in \mathbb{N}} \widetilde{f}^n (G). 
\end{equation*}

\begin{definition}
We say that $G$ has \textit{Bounded geometry} with respect to $\widetilde{f}$ if 

\begin{enumerate}[label={\bfseries [BG\arabic*]}]
\item\label{BG1} for every $n$, the set $\widetilde{f}^n(G)$ is a geometrical wiring, and the corresponding foliations $W^1_{\widetilde{f}^n(G)},..., W^{\ell}_{\widetilde{f}^n(G)}$ are invariant;
\item\label{BG2} The foliations $W^1, ... , W^{\ell}$ are uniformly continuous;
\item\label{BG3} The family $\{ \widetilde{f}: W^i(\widetilde{y}) \to W^i(\widetilde{f}(\widetilde{y})): \widetilde{y} \in G_{\infty} \}$ is uniformly Lipschitz;
\item \label{BD4} for every $\widetilde{x} \in G_{\infty}$, there exists a parametrization $h^i_x: [0,1] \to W^i(\widetilde{x})$ such that the family

\begin{equation*}
    \mathcal{P} := \{ h^i_x : [0,1] \to W^i(\widetilde{x}) : 1 \leq i \leq \ell, \widetilde{x} \in G_{\infty} \}.
\end{equation*}

is uniformly bi-Lipschitz.
\end{enumerate}

\end{definition}

%===============================================

Define $\Gamma^c := \cap_{n=1}^{\infty} \widetilde{f}^{-n}(\widetilde{W}^c_{\alpha}(\widetilde{x}^n))$ as the set of points in $\widetilde{W}^c_{\alpha}(\widetilde{x}^n)$  for which the $n$-iterate are in $\widetilde{W}^c_{\alpha}(\widetilde{x}^n)$. By the geometry of $\widetilde{W}^c_{\alpha}(\widetilde{x})$ (since $E^c$ is simple), it is subfoliated. Therefore $\Gamma^c$ has bounded geometry, since we are taking a local manifold.

%-----------------------------------------------------------------------------------------------

\begin{lemma}
    $h(\widetilde{f}, \Gamma^c)=0$.
\end{lemma} 

\begin{re}
\label{onedimc}
    The Lemma above for the one dimensional case is obvious since every  iterate of an arc $\gamma \subset \widetilde{W}_{\alpha}^i(\tilde{x})$ inside $\Gamma^c$ is uniformly bounded. Then
\end{re}

%--------------------------------------------------------------------

\begin{equation}
\label{ent0}
    h(\widetilde{f}, \gamma)=0,
\end{equation}

for any $\gamma \subset \widetilde{W}_{\alpha}^i(\widetilde{x})$ inside $\Gamma^c$.
\begin{proof}
    Since \textcolor{blue}{$\Gamma^c$} has bounded geometry and $\widetilde{f}$ is a homeomorphism we can use Theorem 4.1 from \cite{carrasco2021invariance} in each $\widetilde{B}(\widetilde{x},\alpha)$, which proves the lemma.
\end{proof}

\begin{re}
    In the paper of P. Carrasco et. al. \cite{carrasco2021invariance} they use the strongly simple property to obtain a good geometric behavior (bounded geometry) of the fiber $h^{-1}(x)$, where $h$ is the semiconjugation of the Derived of Anosov. This behavior is used to prove that $h^{-1}(x)$ generates no entropy. In our case \textcolor{blue}{$\Gamma^c$} has bounded geometry, therefore we can use that result.
\end{re}

\begin{lemma}
    $h(\tilde{f}, \widetilde{D}_{\alpha}^{cs})=0$ .
\end{lemma} 

\begin{proof}
    For this we just need to prove $h(\widetilde{f}, \widetilde{D}_{\alpha}^{cs}(\widetilde{x})) = h(\widetilde{f}, \Gamma^c)$. Let $S$ be an $(n,\epsilon/2)$-spanning set for $\Gamma^c$, is also an $(n,\epsilon)$-spanning set for 
    $$B(\Gamma^c, \epsilon/2) := \{ \widetilde{x} \in M_f : \exists \:\widetilde{y} \in \Gamma^c , d(\widetilde{x}, \widetilde{y}) < \epsilon /2 \}.$$  
For $n \in \mathbb{N}$ sufficiently large, $\widetilde{f}^n(\widetilde{D}^{cs}_{\alpha}(x)) \subset B(\widetilde{f}^n(\Gamma^c), \epsilon/2)$, thus $S$ will be an $(n,\epsilon)$-spanning set for $\widetilde{D}^{cs}_{\alpha}(x)$.

Therefore,
\begin{eqnarray*}
%  \begin{aligned}
h(\widetilde{f}, \Gamma^c) & = & \lim_{\epsilon \rightarrow 0} \limsup_{n 
\rightarrow \infty } \dfrac{1}{n} \log \min \{ \# S : S \subseteq \Gamma^c \hbox{ is an } (n, \epsilon)-\hbox{spanning set} \} \\
 & \geq & \lim_{\epsilon \rightarrow 0} \limsup_{n 
\rightarrow \infty } \dfrac{1}{n} \log \min \{ \# S : S \subseteq \widetilde{D}^{cs}_{\alpha}(x) \hbox{ is an } (n, \epsilon /2)-\hbox{spanning set} \} \\
 & \geq & h(\widetilde{f},\widetilde{D}^{cs}_{\alpha}(x)).
%\end{aligned}  
\end{eqnarray*}

The equality holds since $\Gamma^c \subset \widetilde{D}^{cs}_{\alpha}(x)$. Additionally, by the preceding lemma, we know that $h(\widetilde{f}, \Gamma^c) = 0$, which completes the proof of this lemma.
   
\end{proof}

%---------------------------------------------------------------------
\begin{proof}[Proof of Theorem \ref{theo_hexp1}]

%------------------------------------------------------------------------
By Lemma \ref{lemma_inc}, we know that $\Gamma_{\epsilon}(\widetilde{x}) \subset \widetilde{D}^{cs}_{\alpha}(\widetilde{x})$. Combining this with the previous lemma, we obtain:
\[
0 = h(\widetilde{f}, \Gamma^c) = h(\widetilde{f}, \widetilde{D}^{cs}_{\alpha}(\widetilde{x})) \geq h(\widetilde{f}, \Gamma_{\epsilon}(\widetilde{x})).
\]

which implies that $\widetilde{f}$ is asymptotically $h$-expansive.
\end{proof}

\begin{teoa}
\label{teooa}
   Let $f: M \rightarrow M$ be a absolutely partially hyperbolic endomorphism with simple center bundle, then there exist an equilibrium state for $(f, \phi)$ for any continuous potential $\phi: X \to \R$.  
\end{teoa}

\begin{proof}
    From the previous theorem we have that $h_{\widetilde{\mu}}(\widetilde{f})$ is upper semi continuous by Theorem \ref{semicont}. Furthermore,  setting  $ \widetilde{\phi} := \phi \circ p $  where $p : \widetilde{X} \to X$ is the natural projection,
    
\begin{equation*}
    \widetilde{\mu} \mapsto h_{\widetilde{\mu}}(\widetilde{f}) + \int_X \widetilde{\phi} \, d\widetilde{\mu}
\end{equation*}
is upper semi continuous.    
Then it attains the maximum, an equilibrium state, $\widetilde{\mu}$ for $(\widetilde{f}, \widetilde{\phi})$. Since by Proposition \ref{bijm} there is an injective correspondence between $\mathcal{M}_{\widetilde{f}}(\widetilde{X})$ and $\mathcal{M}_f(X)$, then, by Proposition \ref{meas} the entropy function $\nu \to h_{\nu}(f)$ is also upper semi continuous, then there exist an equilibrium state for $(f,\phi)$. 
\end{proof}

%--------------------------------------------------------------------------

%------------------------------------------------------------------------
\section{Proof of Theorem B}
In this case, we do not require the integrability of the center bundle, since the domination of the splitting will ensure that $\Gamma_{\delta}(\widetilde{x})$ lies within a one-dimensional leaf tangent to the one-dimensional center direction.

To prove Theorem B, we follow an argument similar to that in \cite{Lorenzo2012entropy}, where they prove that $f$ is $h$-expansive. This property follows as a consequence of the fact that the dynamical ball will be inside a center curve. For this, we need to ensure that, for each center curve $\widetilde{\gamma}_i(\widetilde{x})$, the fake foliations $\widehat{W}^{cs,i-1}$ and $\widehat{W}^{cu,i+1}$ exhibit contracting and expanding behavior, respectively, along $\widetilde{\gamma}_i$. 

To achieve this, we apply a Pliss Lemma argument within each fiber $\widetilde{B}(\widetilde{x}, r)$. This is justified by the definition of the tangent bundle $T\widetilde{M}$, which guarantees good behavior of the invariant subbundles along the past orbit of $\widetilde{x}$. Once this structure is established, the argument from \cite{Lorenzo2012entropy} applies directly to $\widetilde{f}$. 

In particular, it follows that the dynamical ball $\Gamma_{\delta}(\widetilde{x})$ is contained in the center directions of the fake foliations $\widehat{W}^{cu,0}_{\widetilde{x}}(\widetilde{x}) \cap \widehat{W}^{cs,k}_{\widetilde{x}}(\widetilde{x})$, and any point $\widetilde{y}$ in the dynamical ball must lie on the same center curve, due to the domination.

Having outlined the main ideas, we now give the formal proof.

From the construction of the tangent bundle $T \widetilde{M}$, we have a good behavior of the subspaces along the past for $\widetilde{M}$. 
We will use the arguments of \cite{Lorenzo2012entropy}.

\begin{re}{(Choice of constants)} We fix some constants:
\label{chconst}
\end{re}

\begin{enumerate}

    \item Fix $\tau>0$ such that $(1+\tau) \sqrt{\lambda} < 1$, where $\lambda < 1$ is the domination constant in (\ref{ds}).
    
    \item Fix $\nu >0$ sufficiently small such that if $\widetilde{y}, \widetilde{y}' \in \widetilde{B}_{5 \nu} (\widetilde{x})$ for some $\widetilde{x} \in \widetilde{M}$, then for all $i \in \{ 0,...,k \}$ it holds

\begin{equation*}
    1-\tau < \dfrac{||Df^{-1}|_{\bar{E}^{cu,i}(y_n)}||}{||Df^{-1}|_{\bar{E}^{cu,i}(y'_n)}||} < 1 + \tau \: \hbox{ and } \; 1-\tau < \dfrac{||Df|_{\bar{E}^{cs,i}(y_n)}||}{||Df|_{\bar{E}^{cs,i}(y'_n)}||} < 1 + \tau.
\end{equation*}    
\end{enumerate}

One of the key ingredients leading to the fact that $\Gamma_{\delta}(\widetilde{x})$ lies in a one-dimensional leaf is that the dominated splitting guarantees the existence of hyperbolic times. To establish this, we employ a reformulation of the Pliss Lemma.
%stated for the sets $\Lambda_V$ satisfying the hypotheses of theorem ***

\begin{lemma}{\cite{Pliss72}}
\label{thip1}
    Let $\lambda >0$ be the constant of domination and $0< \lambda < \lambda_1 < \lambda_2 < 1$. Assume that $\widetilde{x}=(x_n)_n$, $i \in {0,..,k}$ and there exist $ n' \geq 0$ such that

\begin{equation}
    \prod_{m=0}^{n'} ||Df |_{E^{cs,i}(x_m)}|| \leq \lambda_1^{n'}
\end{equation}

Then there are $N=N(\lambda_1, \lambda_2, f) \in \N$ and a constant $c= c(\lambda_1, \lambda_2 , f)>0$ such that for every $n \geq N$ there exist $\ell \geq cn$ and numbers (hyperbolic times)

\begin{equation*}
    0 < n_1 < n_2 < \cdots < n_{\ell} < n
\end{equation*}
such that 

\begin{equation*}
    \prod_{m=n_r}^h ||Df|_{E^{cs,i}(x_m)}|| \leq \lambda_2^{h-n_r},
\end{equation*}
for all $r=1,2,..., \ell$ and all $h$ with $n_r \leq h \leq n$. Similar assertions folds for the map $\widetilde{f}^{-1}$ and the bundles $E^{cu,i}$.
\end{lemma}

The following application of the previous lemma gives a lower bound for the expansion of $Df$ along the subspace $E^{cs,i}$. Let $\gamma_i(\widetilde{x})$ be a central curve in $ \widetilde{B}(\widetilde{x},\delta)$ and $\widetilde{y}, \widetilde{z} \in \gamma_i(\widetilde{x})$, we denote $[\widetilde{y},\widetilde{x}]_{\gamma_i(\widetilde{x})}$ as the segment in $\gamma_i(\widetilde{x})$ with endpoints $\widetilde{y}$ and $\widetilde{z}$.

\begin{lemma}
    Consider a small enough $\delta>0$. If $\widetilde{x} \in \widetilde{M}$, $\widetilde{y} \in \Gamma_{\delta}(\widetilde{x})$, $ \widetilde{y} \neq \widetilde{x}$, $i \in \{1,...,k \}$. Suppose that $\widetilde{y} \in \gamma_i(\widetilde{x})$. If $\lambda_1 \in (\lambda, \sqrt{\lambda})$, then there is $n_0>0$ such that 

\begin{equation*}
    \prod_{j=0}^n||Df|_{E^{cs,i}(\widetilde{f}^j(\widetilde{f}^{-n}(\widetilde{y}')))}||> \lambda_1^n \hbox{ and } \prod_{j=0}^n||Df^{-1}|_{E^{cu,i}(\widetilde{f}^j(\widetilde{y}'))}||> \lambda_1^n
\end{equation*}

for all $\widetilde{y}' \in [\widetilde{x},\widetilde{y}]_{\gamma_i(\widetilde{x})}$ and $n>n_0$.
    
\end{lemma}

\begin{proof}

We will restrict to the $E^{cu,i}$ case, since for the  another case $E^{cs,i}$, is very similar than \cite{Lorenzo2012entropy} . 

    Let $\lambda_2 \in (\lambda_1,1)$ be such that $(1+\tau)\lambda_2 <1$ where $\tau$ is as in Remark \ref{chconst}. There exist infinitely many $m_n \in \N$, $m_n \to \infty$, and a sequence in $\widetilde{M}$, $\widetilde{y}^n \in [\widetilde{x},\widetilde{y}]_{\gamma_i(\widetilde{x})}$ such that

\begin{equation*}
    \prod_{j=0}^{m_n} ||Df^{-1}|_{E^{cu,i}(y^n_{m_n-j}))}|| \leq \lambda_1^n.
\end{equation*}

where $Df^{-1}|_{E^{cu,i}_{y_k}}$ is defined for each sequence $\widetilde{y}^n$ in $\widetilde{M}$.

By Lemma \ref{thip1}, we have

\begin{equation*}
    \prod_{j=0}^{m_n} ||Df^{-1}|_{E^{cu,i}(y^n_{m_n - j}))}|| \leq \lambda_2^{m_n}.
\end{equation*}
For all $j \geq 0$, the curve $\widetilde{f}^{j}([\widetilde{x},\widetilde{y}]_{\gamma_i(\widetilde{x})})$ stays $2\epsilon$-close to $\widetilde{f}^{j}(x)$, and $2\epsilon$-close to $\widetilde{f}^{j}(y^n)$. From remark \ref{chconst}, for all $\widetilde{y}' \in \widetilde{f}^{m_n}([\widetilde{x},\widetilde{y}]_{\gamma_i(\widetilde{x})})$, one has 

\begin{equation*}
        \prod_{j=0}^{m_n} ||Df^{-1}|_{E^{cu,i}(y'_{-j})}|| \leq ((1+\tau)\lambda_2)^{m_n}.
\end{equation*}
In particular

\begin{equation*}
    \ell([\widetilde{x},\widetilde{y}]_{\gamma_i(\widetilde{x})}) \leq ((1+\tau)\lambda_2)^{m_n} \ell([\widetilde{f}^{m_n}(\widetilde{x}),\widetilde{f}^{m_n}(\widetilde{y})]_{\gamma_i(\widetilde{f}^{m_n}(\widetilde{x}))}) 
\end{equation*}

If $\epsilon$ is small enough, then  $\ell([\widetilde{f}^{m_n}(\widetilde{x}),\widetilde{f}^{m_n}(\widetilde{y})]_{\gamma_i(\widetilde{f}^{m_n}(\widetilde{x}))})$ is bounded by $2\delta$. Thus, by letting $m_n \to \infty$, then $\widetilde{x}=\widetilde{y}$, which is a contradiction.
    \end{proof}

This lemma means that, for $\delta>0$ small enough, if $\widetilde{y}' \in [\widetilde{x}, \widetilde{y}]_{\gamma_i(\widetilde{x})}$ and $\widetilde{y} \neq \widetilde{x}$, the leaves of the fake foliations $\widehat{W}^{cs,i-1}(\widetilde{y}')$ and $\widehat{W}^{cu,i+1}(\widetilde{y}')$ behave as contraction and expansion leaves along $\widetilde{B}(\widetilde{x},\delta)$. We have as a consequence:

\begin{cor}
\label{cordom}
    Let $\widetilde{y} \in \Gamma_{\delta}(\widetilde{x}) \setminus \{ \widetilde{x} \}$. Then for all $y' \in [\widetilde{x},\widetilde{y}]_{\gamma_i(\widetilde{x})}$ we have that

\begin{equation*}
   \widehat{W}^{cu,i+1}_{\widetilde{x}}(\widetilde{y}') \cap \Gamma_{\delta}(\widetilde{x})= \{ \widetilde{y}' \} \hbox{ and } \widehat{W}^{cs,i-1}_{\widetilde{x}}(\widetilde{y}') \cap \Gamma_{\delta}(\widetilde{x})= \{ \widetilde{y}' \}
\end{equation*}
    
\end{cor}

\begin{prop}
\label{ccurve}
    For $\delta>0$ sufficiently small, the set $\Gamma_{\delta}(\widetilde{x})$ is either $\{ \widetilde{x} \}$ or is contained in a curve $\widetilde{\gamma}_i(\widetilde{x})$ for some $i \in \{ 1,...,k \}$.
\end{prop}

\begin{proof}
    Since $\widetilde{f}: \widetilde{M} \to \widetilde{M}$ is an homeomorphism, the proof is a direct consequence of  Corollary \ref{cordom} together with the same arguments of \cite{Lorenzo2012entropy}.
\end{proof}

\begin{proof}[Proof of Theorem B]
The previous Proposition implies that $h(\widetilde{f}, \Gamma_{\delta}(\widetilde{x}) )=0$ following the same ideas presented in Remark \ref{onedimc}, this implies that $\widetilde{f}$ is $h$-expansive. Moreover, by Theorem \ref{semicont}, the map $\widetilde{f}$ is upper semi-continuous. Applying the same arguments used in the final part of Theorem A, we conclude that the pair $(f, \phi)$ admits an equilibrium state.
\end{proof}

\section{Proof of Theorem C}

First we prove that $\lambda^c(\tilde{\mu})$ is continuous. Let $\tilde{\mu} \mapsto \lambda^c(\tilde{\mu})$ be the central Lyapunov exponent for $\tilde{f}$, where $\tilde{\mu}$ is an $\tilde{f}$-invariant measure in $\widetilde{M}$.
%DEFINIR Lyapunov EXPONENT FOR F.
\begin{lemma}

The center Lyapunov exponent for $\tilde{f}$ satisfies the following equation:
    \begin{equation}
        \lambda^c(\tilde{\mu}) = \int_{\widetilde{M}} \log ||Df|_{E^c(\tilde{x})}|| d{\tilde{\mu}},  
    \end{equation}
this implies that $\lambda^c(\tilde{\mu})$ is continuous in $\mathcal{M}(\tilde{f})$.    
\end{lemma}

\begin{proof}
Since the dimension of the center subspace $E^c_{\tilde{x}}$ is one dimensional and by the definition of $D\tilde{f}_{\tilde{x}}$ we have the following:

    \begin{eqnarray*}
\lambda^c(\tilde{\mu}) & = & \int_{\widetilde{M}} \lim_{n \to \infty} \dfrac{1}{n} \log ||D\tilde{f}^n|_{E^c(\tilde{x})}|| d\tilde{\mu} \\
& = & \int_{\widetilde{M}} \lim_{n \to \infty} \dfrac{1}{n} \log ||Df^n|_{E^c(x_0)}|| d\tilde{\mu} \\
  & = & \int_{\widetilde{M}} \lim_{n \to \infty} \dfrac{1}{n} \sum_{j=0}^{n-1} \log ||Df|_{E^c_j(f^j(x_0))}|| d\tilde{\mu}  \\
  & = & \int_{\widetilde{M}} \log ||Df|_{E^c_j(x)}|| d\tilde{\mu}(\tilde{x}), \hbox{by Birkhoff's ergodic theorem}.
\end{eqnarray*}

Since $f \in C^{1+\alpha}(M)$ and a local diffeomorphism, the continuity follows. 

\end{proof}

Let $\mu_n$ be a family of ergodic measures of maximal entropy of $f$ such that $\mu_n \to \mu$, We have the following upper semicontinuity inequality given in \cite{misiurewicz1976topological}. %Theorem 4.2].

    \begin{equation}
    \label{tail_ineq_1}
        \limsup_{ n \to \infty} h_{\mu_n}(f) \leq h_{\mu}(f) + h^*(f) \hbox{ if } \mu_n \longrightarrow \mu \hbox{ weak-}*,
    \end{equation}
where $h^*(f) := \lim_{\epsilon \to 0^+} h^*_f(\epsilon)$ is the tail entropy. We know by \cite{alvarez2022existence} that $f$ is assymptotically expansive, therefore $h^*(f)=0$, then by (\ref{tail_ineq_1}) $\mu$ maximize the entropy of $f$.

Since $\lambda^c(\tilde{\mu}) \neq 0$, we consider two cases:

\textbf{Fist case:} When $\lambda^c(\tilde{\mu}) < 0$, the argument closely follows that of \cite{lima2024measures}, as the condition imposed by the authors, $h_{\mathrm{top}}(f) > \log(\deg(f))$, ensures a negative central Lyapunov exponent. Therefore, we focus on the remaining case, where $\lambda^c(\tilde{\mu}) > 0$.

\textbf{Second case:}  Assume $\lambda^c(\tilde{\mu})>0$. We lift the measures $\mu_n$ and $\mu$, where $\tilde{\mu}_n \to \tilde{\mu}$. By the continuity of $\lambda^c$ there exist $n_0$, such that $\lambda^c(\tilde{\mu}_n)$ is inside a small neighborhood of $\lambda^c(\tilde{\mu})$, then without loss of generality there exist $\chi>0$ such that the central Lyapunov exponent of $\tilde{\mu}_n$ for $\tilde{f}^{-1}$ is smaller than $-\chi$ for $n \geq n_0$.

Take $\chi>0$ such that the central Lyapunov exponent of $\tilde{\mu}_n$ for $\tilde{f}^{-1}$ is smaller than $-\chi$ for $n \geq n_0$. Let $\Lambda^u$ be the sets of points where we look for uniform expansion, defined as:

\begin{equation*}
    \Lambda^u := \left\{ (x_n) \in \widetilde{M}: \log ||Df^{-1}_{x_{-n}} \cdot ... \cdot Df^{-1}_{x_0}|| < \frac{-n \chi}{2} \right\}
\end{equation*}

\begin{lemma}
\label{pliss_0}
    There is $\delta>0$ such that $\tilde{\mu}_n(\Lambda^u)>\delta$ for $n \geq n_0$
\end{lemma}

Using a version of the Pliss Lemma provided in \cite{Crovisier2018dissipative}, we show that every ergodic measure of maximal entropy assigns positive measure to the set $\Lambda^u$.

\begin{lemma}[(\cite{Crovisier2018dissipative}]
\label{pliss}
    For any $\epsilon >0$, $\alpha_1 < \alpha_2$ and any sequence $(a_i) \in (\alpha_1,+\infty)^{\N}$ satisfying

\begin{equation*}
    \limsup_{n \to +\infty} \dfrac{a_0 + \cdots + a_{n-1}}{n} \leq \alpha_2,
\end{equation*}

there exists a sequence of integers $0<n_1<n_2< \cdots$ such that:

\begin{itemize}
    \item for any $\ell \geq 1$ it holds
    \begin{equation*}
        \dfrac{a_{n_{\ell}}+ \cdots + a_{n-1}}{(n-n_{\ell})} \leq \alpha_2 + \epsilon, \, \forall n > n_{\ell}.
    \end{equation*}
    \item the upper density $\limsup_{\ell \to +\infty} \frac{\ell}{n_{\ell}}$ is larger than $\delta = \frac{\epsilon}{\alpha_2+\epsilon-\alpha_1}$.
\end{itemize}
    
\end{lemma}

\begin{proof}[Proof of lemma \ref{pliss_0}]
    Let $\tilde{\mu}_n$ be an ergodic measure of maximal entropy for $\tilde{f}$, such that $n \geq n_0$. 
For any $\tilde{x} = (x_n)_n$ and each $i \geq 0$, set $a_i = \log|| Df^{-1}_{x_{-i}}|_{E^c}||$, and let 

\begin{equation*}
    \alpha_1 = \inf_{\tilde{x} \in \widetilde{M}} \log ||Df^{-1}_{x_0}|_{E^c}||,\, \alpha_2 = -\chi,\, \epsilon = \frac{\chi}{2}.
\end{equation*}

By lemma \ref{pliss} exists $(n_{\ell})_{\ell}$ such that for any $\ell \geq 0$ and $N \geq n_{\ell}$, it follows

\begin{equation*}
        \dfrac{\log || Df^{-1}_{x_{-N+1}}|_{E^c} \cdot ... \cdot Df^{-1}_{x_{-n_{\ell}}}|_{E^c} ||}{N-n_{\ell}} =     \frac{\log ||Df^{-1}_{x_{-N+1}}|_{E^c}|| + \cdots +\log ||Df^{-1}_{x_{-n_{\ell}}}|_{E^c}|| }{N-n_{\ell}} \leq  -\frac{\chi}{2}
\end{equation*}

which implies that for any $\tilde{x} \in \widetilde{M}$, $\tilde{f}^{-n_{\ell}}(\tilde{x}) \in \Lambda^u$. By lemma \ref{pliss} and Birkhoff's ergodic theorem, we obtain that

\begin{equation*}
  \tilde{\mu}_n(\Lambda^u) \geq \limsup_{\ell \to \infty} \frac{\ell}{n_{\ell}}>\delta.   
\end{equation*}

\end{proof}

\begin{lemma}
\label{unif_u}
    There exist $\ell_0>0$ such that the length of the local unstable manifold $W^u(x_n)$ is larger than $\ell_0$ for any $(x_n)_n \in \Lambda^u$
\end{lemma}

\begin{proof}
    
Since $Df$ is uniformly continuous in the projective tangent bundle, there exist a small center cone field $\mathcal{C}^c$ and $\delta_0>0$ such that:

\begin{equation*}
    d((x_n), (y_n)) < \delta_0 , v \in \mathcal{C}^c_{y_0} \Rightarrow | \log||Df^{-1}_{x_0}|_{E^c_{x_0}}||-\log||Df^{-1}_{y_0}v|| | < \frac{\chi}{4}.
\end{equation*}
This implies $||Df^{-1}_{y_0}v|| < e^{\frac{\chi}{4}}||Df^{-1}_{x_0}|_{E^c_{x_0}}||$, the bundle $E^c$ is locally integrable, i.e. there is $\delta_1>0$ and a continuous family of curves $\gamma^c_{x_n}$ centered at $(x_n)$ and tangent to $E^c$ such that $|\gamma^c_{x_n}|> \delta_1$ for every $(x_n) \in \widetilde{M}$.

Let $\ell_0 := \min\{ \delta_0 , \delta_1 \}$, define $W^c((x_n), \ell_0)$ as the subset of $\gamma^c_{x_n}$. We claim that if $(x_n) \in \Lambda^u$, then $W^c((x_n), \ell_0)$ is contained in the unstable manifold of $(x_n)$. For this, fix $(x_n) \in \Lambda^u$ and let $\gamma(t)$ be a parametrization of $W^c((x_n),\ell_0)$. We have

%Definir bien $\gamma(t)$ para el caso de la extension de M.

\begin{equation*}
    ||Df^{-1}_{\gamma(t)} \gamma'(t)|| < 
    e^{\frac{\chi}{4}} ||Df^{-1}_{x_0}|
    _{E^c_{x_0}}||||\gamma'(t)||,
\end{equation*}
and so $|\tilde{f}^{-1}(W^c((x_n),\ell_0))|< e^{-\frac{\chi}{4}} \ell_0 < \delta_0$. By induction, it follows that $|\tilde{f}^{-n} (W^c((x_n),\ell_0)) | < e^{-\frac{n\chi}{4}} \ell_0$ for all $n \geq 0$, this implies that $W^c((x_n),\ell_0)$ expands exponentially fast.
    
\end{proof}

Moreover, for any $(x_n)_n, (y_n)_n \in \Lambda^u$, the unstable and stable manifolds satisfy $W^u((x_n)) \pitchfork W^s((y_n)) \neq \emptyset$, by Lemma \ref{unif_u} and the dominated splitting of $T\tilde{M}$. By \cite[Theorem A]{lima2024measures}, it follows that $\tilde{\mu}_n \sim \tilde{\mu}_m$ for all $n, m \geq n_0$. Since each equivalence class contains a unique measure of maximal entropy, the desired result follows immediately.

\bibliographystyle{abbrv}
\bibliography{Bibl.bib}
	
\end{document}